\documentclass{birkjour}
 \newtheorem{thm}{Theorem}[section]
 
 \newtheorem{lem}[thm]{Lemma}
 
 \theoremstyle{definition}
 \newtheorem{defn}[thm]{Definition}
 \theoremstyle{remark}
 
 \newtheorem*{ex}{Example}
 \numberwithin{equation}{section}

\begin{document}

%-------------------------------------------------------------------------
% editorial commands: to be inserted by the editorial office
%
%\firstpage{1} \volume{228} \Copyrightyear{2004} \DOI{003-0001}
%
%
%\seriesextra{Just an add-on}
%\seriesextraline{This is the Concrete Title of this Book\br H.E. R and S.T.C. W, Eds.}
%
% for journals:
%
%\firstpage{1}
%\issuenumber{1}
%\Volumeandyear{1 (2004)}
%\Copyrightyear{2004}
%\DOI{003-xxxx-y}
%\Signet
%\commby{inhouse}
%\submitted{March 14, 2003}
%\received{March 16, 2000}
%\revised{June 1, 2000}
%\accepted{July 22, 2000}
%
%
%
%---------------------------------------------------------------------------
%Insert here the title, affiliations and abstract:
%

\title[Compact ADI method for two dimensional RSFDE]
 {Compact ADI method for two-dimensional\\ Riesz space fractional diffusion equation}

%----------Author 1
\author[Sohrab Valizadeh]{Sohrab Valizadeh}

\address{%
Department of Mathematics\\
Faculty of Sciences\\
University of Mohaghegh Ardabili\\
P. O. Box: 56199-179\\
Ardabil\\
Iran}

\email{valizadeh.s@uma.ac.ir}

%\thanks{This work was completed with the support of our \TeX-pert.}
%----------Author 2
\author{Alaeddin Malek}
\address{Department of Mathematics\br
Tarbiat Modares University\br
P. O. Box: 14115-175\br
Tehran\br
Iran}
\email{mala@modares.ac.ir}

%----------Author 3
\author{Abdollah Borhanifar}
\address{Department of Mathematics\br
Faculty of Sciences\br
University of Mohaghegh Ardabili\br
P. O. Box: 56199-179\br
Ardabil\br
Iran}
\email{borhani@uma.ac.ir}

%----------classification, keywords, date


\date{January 1, 2004}
%----------additions
%\dedicatory{To my boss}
%%% ----------------------------------------------------------------------

\begin{abstract}
In this paper, a compact alternating direction implicit (ADI) method has been developed for solving two-dimensional Riesz space fractional diffusion equation. The precision of the discretization method used in spatial directions is twice the order of the corresponding fractional derivatives. It is proved that the proposed method is unconditionally stable via the matrix analysis method and the maximum error in achieving convergence is discussed. Numerical example is considered aiming to demonstrate the validity and applicability of the proposed technique.
\end{abstract}

%%% ----------------------------------------------------------------------
\maketitle
%%% ----------------------------------------------------------------------
%\tableofcontents
\section{Introduction}
Fractional calculus is a natural extension of the integer order calculus \cite{Miller1993,Podlubny1998}. Recently, many problems in physics \cite{Hilfer2000,Sagi2012}, biology \cite{Gal2010,Magin2006}, finance \cite{Cartea2007749} and hydrology \cite{Baeumer20011543,Benson20001403,Dentz2004155} have been formulated on fractional partial differential equations, containing derivatives of fractional order in space, time or both. Fractional derivatives play a key role in modelling particle transport in anomalous diffusion. The space fractional diffusion equation describes L\'{e}vy flights \cite{Benson20001403,Liu2004209}. The time fractional diffusion equation depicts traps, and the time-space fractional diffusion equation characterizes the competition between L\'{e}vy flights and traps \cite{Zaslavsky2002461}. The regularity criterion is important for the diffusion equations that are proposed in dynamic systems. Sadek et al. \cite{Sadek20121741} established the Serrin-type regularity criteria for the 3D nematic liquid crystal flows in the terms of the multiplier space $X_{r}^{R^3}$. Sadek and Ragusa \cite{Sadek20161271} studied the regularity criterion in terms of the homogeneous Besov space for the incompressible Boussinesq equations.

Numerical methods to different types of fractional diffusion models are increasingly appearing in the sciences. Some models can be mentioned, which are groundwater flow, the flow of heat transfer in furnaces, dissolving gases in liquids and fluid flow in a porous medium. In addition, the analytical solutions of such equations are usually difficult to obtain, so in order to gain their numerical solutions becomes more important and emergent. For one-dimensional problems, Cui \cite{Cui20097792} extended a compact finite difference method for the fractional diffusion equation with the Riemann-Liouville derivative via the Gr\"{u}nwald-Letnikov discretization. Alikhanov \cite{Alikhanov2015424} constructed a widespread difference approximation of the Caputo fractional derivative for the time fractional diffusion equation with variable coefficients. Borhanifar and Valizadeh \cite{Borhanifar2015466} considered Mittag-Leffler-Pad\'{e} approximations for space and time fractional diffusion equations by using shifted Gr\"{u}nwald estimate in space, rational recurrence formula in time, and discussed their stabilities and truncation errors. Celik and Duman \cite{Celik20121743} used the fractional centered difference that introduced by Ortigueira \cite{Ortigueira20061} to solve the Riesz fractional diffusion equation and also for this type equation, some of the authors employed the matrix stemming from the discretization of the Riesz space derivative by compact difference scheme and parameter
spline function \cite{Zhang2014266} and fractional centered difference formula \cite{Popolizio20131975}. For two-dimensional problems, Bu et al. \cite{Bu201426} developed the Galerkin finite element method for the numerical study of the two-dimensional Riesz space fractional diffusion equations combined with a backward difference method. Tadjeran and Meerschaert \cite{Tadjeran2007813} applied a mixed Crank-Nicolson- ADI method not only as a discretization, but also as a Richardson extrapolation to obtain a numerical solution of the two-dimensional space fractional diffusion equation and they examined of being unconditionally stable and second order accuracy of the method. H. Wang and K. Wang \cite{Wang20117830} investigated an $O(N log^{2} N)$ alternating-direction finite difference method for the two dimensional fractional diffusion equations and at the same time, Zhang and Sun \cite{Zhang20118713} explored ADI schemes for the two-dimensional fractional sub-diffusion equation. Zeng et al. \cite{Zeng20142599} derived approximate solution via Crank-Nicolson ADI spectral method for the two-dimensional Riesz space fractional nonlinear reaction diffusion equation. Gao and Sun \cite{Gao2015487} handled high order compact ADI schemes for the 2D time-fractional advection-diffusion equation.
%\cite{Rahman2014264}
The numerical solutions of the two-dimensional Riesz space fractional diffusion equations have been challenging. The main purpose of this paper is to solve the two-dimensional Riesz space fractional diffusion equations using the compact difference scheme with the operator-splitting techniques, that is, using the compact ADI scheme.

Let $\Omega$ be a rectangular domain in $\emph{R}^2$ with boundary $\Gamma=\partial \Omega$ and $J=(0,T]$ be the time interval, $T > 0$. In this paper, we consider the following two dimensional Riesz space fractional diffusion equation for a solute concentration $u$

\begin{eqnarray*}\label{math201801-01}
\frac{\partial u(x,y,t)}{\partial
t}=\mathcal{C}_{x}\frac{\partial^{\alpha}u(x,y,t)}{\partial
|x|^{\alpha}}+\mathcal{C}_{y}\frac{\partial^{\beta}u(x,y,t)}{\partial
|y|^{\beta}}+s(x,y,t)%,\quad (x,y,t)\in\Omega\times J,
\end{eqnarray*}
\begin{eqnarray}\label{math201801-01}
(x,y,t)\in\Omega\times J,
\end{eqnarray}
\begin{eqnarray}\label{math201801-02}
u(x,y,0)=f(x,y), \quad (x,y)\in\Omega,
\end{eqnarray}
\begin{eqnarray}\label{math201801-03}
u(x,y,t)=0, \quad (x,y,t)\in \Gamma\times J,
\end{eqnarray}
where $\mathcal{C}_{x}$ and $\mathcal{C}_{y}$ are the average fluid velocities in the x- and y-directions. We restrict $1 < \alpha, \beta \leq 2$ and assume $\mathcal{C}_{x},\mathcal{C}_{y}\geq 0$. The solution $u=u(x,y,t)$ is assumed to be sufficiently smooth and has the necessary continuous partial derivatives up to certain orders.

The outline of the paper is organized as follows. Preliminaries and basic definitions are presented in the next section. Section 3 is devoted to the construction and explanation of numerical algorithm that the Crank-Nicolson scheme and the alternating directions implicit method is combined together. In Section 4, the stability and the convergence order of the numerical scheme are theoretically analyzed. One example is given in Section 5 and some conclusions are drawn in Section 6.

\section{Preliminaries and basic definitions}

\begin{defn}
The Riesz fractional operator for $n-1<\gamma\leq n$ on a finite interval $a\leq x\leq b$ is defined as \cite{Gorenflo1998167,Kilbas1993}
\begin{eqnarray}\label{math201801-04}
\begin{split}
\frac{\partial^{\gamma}v(x,t)}{\partial
|x|^{\gamma}}=-\vartheta_{\gamma}(_{a}D_{x}^{\gamma}+{}_{x}D_{b}^{\gamma})v(x,t),
\end{split}
\end{eqnarray}
where
\begin{eqnarray*}
\begin{split}
\vartheta_{\gamma}=\frac{1}{2cos(\frac{\pi\gamma}{2})}, \quad \gamma \neq1,
\end{split}
\end{eqnarray*}
\begin{eqnarray*}
\begin{split}
_{a}D_{x}^{\gamma}v(x,t)=\frac{1}{\Gamma(n-\gamma)}\frac{\partial^{n}}{\partial x^{n}}\int_{a}^{x}(x-\xi)^{n-\gamma-1}v(\xi,t)d\xi,
\end{split}
\end{eqnarray*}
\begin{eqnarray*}
\begin{split}
_{x}D_{b}^{\gamma}v(x,t)=\frac{(-1)^n}{\Gamma(n-\gamma)}\frac{\partial^{n}}{\partial x^{n}}\int_{x}^{b}(\xi-x)^{n-\gamma-1}v(\xi,t)d\xi.
\end{split}
\end{eqnarray*}
\end{defn}

\begin{lem}\label{math201801-05}
For a function $h(x)$ defined on the infinite domain $-\infty <x<\infty$, the following equality holds:
\begin{equation}\label{math201801-06}
-(-\Delta)^{\frac{\gamma}{2}}h(x)=-\frac{1}{2cos(\frac{\pi\gamma}{2})}[_{-\infty}D_{x}^{\gamma}h(x)+
{}_{x}D_{\infty}^{\gamma}h(x)]=\frac{\partial^{\gamma}}{\partial
|x|^{\gamma}}h(x).
\end{equation}
\end{lem}
\begin{proof}
See Ref. \cite{Yang2010200}
\end{proof}

\begin{defn}\label{math201801-07}
\cite{Ding20121135} Let the Laplacian $(-\Delta)$ has a complete set of orthonormal eigenfunctions $\varphi_{n}$ corresponding to
eigenvalues $\lambda_{n}^{2}$ on a bounded region $\Omega$ with the homogeneous boundary conditions, then
\begin{center}
 $(-\Delta)^{\frac{\gamma}{2}}f=\left\{%
\begin{array}{ll}
    (-\Delta)^{m}f,   & \gamma=2m,\quad m=0,1,2,..., \\
    (-\Delta)^{\frac{\gamma}{2}-m}(-\Delta)^{m}f,   & m-1<\frac{\gamma}{2}<m,\quad m=1,2,..., \\
    \sum_{n=1}^{\infty}\lambda_{n}^{\gamma}\langle f,\varphi_{n}\rangle \varphi_{n},   & \gamma<0. \\
\end{array}%
\right.$\end{center}
\end{defn}

\begin{lem}\label{math201801-08}
\cite{Meyer2004} The eigenvalues and eigenvectors of the following tridiagonal Toeplitz matrix
\begin{equation*}
A=\left(
\begin{array}{ccccc}
  b & a &   &   &   \\
  c & b & a &   &   \\
    & \ddots & \ddots & \ddots &  \\
   &   & c & b & a   \\
  &   &   & c & b
\end{array}
\right)_{n\times n}
\end{equation*}

are given by

\begin{equation}\label{math201801-09}
\lambda_{j}=b+2a\sqrt{c/a}\cos(j\pi/(n+1)), \quad j=1,2,...,n,
\end{equation}
while the corresponding eigenvectors are:
\begin{equation*}
x_{j}=\left(
\begin{array}{c}
  (c/a)^{1/2}\sin(1j\pi/(n+1)) \\
  (c/a)^{2/2}\sin(2j\pi/(n+1)) \\
  (c/a)^{3/2}\sin(3j\pi/(n+1)) \\
  \vdots \\
  (c/a)^{n/2}\sin(nj\pi/(n+1)) \\
\end{array}
\right), \quad j=1,2,...,n,
\end{equation*}
i.e., $A x_{j}=\lambda_{j} x_{j}$, $j=1,2,...,n$. Moreover, the matrix $A$ is diagonalizable and $P=(x_{1}\quad x_{2}\quad ...\quad x_{n})$ diagonalizes $A$,
i.e., $P^{-1}AP=D$, where $D=diag(\lambda_{1}\quad \lambda_{2}\quad ...\quad \lambda_{n})$.
\end{lem}

\begin{defn}\label{math201801-10}
\cite{Hunter2004} Let $f,g:\mathbb{R}-\{0\} \rightarrow \mathbb{R}$ be real functions. We say $f=\mathcal{O}(g)$ as $x \rightarrow 0$ if there are constants $C$ and $r>0$ such that
\begin{equation*}
\mid f(x)\mid\leq C \mid g(x)\mid \quad \mbox{whenever} \quad 0<\mid x \mid < r.
\end{equation*}
\end{defn}
also following properties of asymptotic estimates are hold for "$\mathcal{O}$" \cite{Erdelyi1956}:
\begin{equation*}
\mathcal{O}(f(x))+\mathcal{O}(f(x))=\mathcal{O}(f(x)),
\end{equation*}
\begin{equation*}
\mathcal{O}(f(x))\mathcal{O}(g(x))=\mathcal{O}(f(x)g(x)).
\end{equation*}

\begin{lem}\label{math201801-11}
If $g(x)$ be a smooth function on $\mathbb{R}$ that is discretized in a finite interval $[a,b]$ include $n$ nodal points as $x_{i}=a+i\mathit{h}$ in which $\mathit{h}=\frac{b-a}{n}$ then $\frac{1}{\mathit{h}^2}\frac{\delta_{x}^2}{1+\frac{\delta_{x}^2}{12}}$  operator approximates the second derivative of the $g(x)$ from the fourth order at inner nodal points of $[a,b]$.
\end{lem}

\begin{proof}
According to being smooth the function $g(x)$, there is continuous function  $f(x)$ that
\begin{equation*}
g''(x)=\frac{d^{2}g(x)}{dx^{2}}=f(x).
\end{equation*}
Alternatively, we can write, $g (x) $ is an exact solution of the above differential equation.
To prove the lemma, needs to show the following relationship is confirmed in the internal nodes $x_{i}$, $i=1,2,...n-1$.
\begin{equation*}
\frac{1}{\mathit{h}^2}\delta_{x}^2g(x_{i})-(1+\frac{\delta_{x}^2}{12})f(x_{i})=\mathcal{O}(\mathit{h}^4),
\end{equation*}
we apply the relevant operators on $g(x_{i})$ and then $f(x_{i})$
\begin{equation}\label{math201801-12}
\frac{1}{\mathit{h}^2}\delta_{x}^2g(x_{i})=\frac{g(x_{i}+h)-2g(x_{i})+g(x_{i}-h)}{\mathit{h}^2},
\end{equation}
\begin{align}\label{math201801-13}
(1+\frac{\delta_{x}^2}{12})f(x_{i})&=(1+\frac{\delta_{x}^2}{12})g''(x_{i})\\
&=g''(x_{i})+\frac{g''(x_{i}+h)-2g''(x_{i})+g''(x_{i}-h)}{12}.%\tag*{\qed}
\end{align}
%\begin{equation}
%=
%g''(x_{i})+\frac{g''(x_{i}+h)-2g''(x_{i})+g''(x_{i}-h)}{12}.
%\end{equation}
By substituting the Taylor series of the function $g(x_{i}+h)$ and $g(x_{i}-h)$ about $x=x_{i}$ in the formula (\ref{math201801-12}) and using the average value theorem for derivatives, we have
\begin{equation*}\label{math201801-14}
\frac{1}{\mathit{h}^{2}}\delta_{x}^2g(x_{i})=g''(x_{i})+\frac{g^{(4)}(x_{i})}{12}\mathit{h}^{2}
+\frac{g^{(6)}(x_{i})}{360}\mathit{h}^{4}+\frac{g^{(8)}(\xi)}{20160}\mathit{h}^{6},
\end{equation*}
\begin{equation}\label{math201801-14}
x_{i}-h<\xi<x_{i}+h.
\end{equation}
Similarly, $g''(x_{i}+h)$ and $g''(x_{i}-h)$ in the formula (\ref{math201801-13}) replace with Taylor expansion theirs centered at $x=x_{i}$ and applying the average value theorem for derivatives, we have
\begin{equation*}\label{math201801-15}
(1+\frac{\delta_{x}^2}{12})f(x_{i})=g''(x_{i})+\frac{1}{12}[g^{(4)}(x_{i})\mathit{h}^{2}+
\frac{g^{(6)}(x_{i})}{12}\mathit{h}^{4}+\frac{g^{(8)}(\zeta)}{360}\mathit{h}^{6}],
\end{equation*}
\begin{equation}\label{math201801-15}
x_{i}-h<\zeta<x_{i}+h.
\end{equation}
The result is following equation by subtracting formula (\ref{math201801-15}) from formula (\ref{math201801-14}), utilized the average value theorem
\begin{equation*}
\frac{1}{\mathit{h}^2}\delta_{x}^2g(x_{i})-(1+\frac{\delta_{x}^2}{12})f(x_{i})=
-\frac{1}{240}g^{(6)}(x_{i})\mathit{h}^{4}-\frac{1}{60480}g^{(8)}(\eta)\mathit{h}^{6},
\end{equation*}
\begin{equation*}
x_{i}-h<\eta<x_{i}+h,
\end{equation*}
therefore
\begin{equation*}
g''(x_{i})=\frac{1}{\mathit{h}^2}\frac{\delta_{x}^2}{1+\frac{\delta_{x}^2}{12}}g(x_{i})+\mathcal{O}(\mathit{h}^4).
\end{equation*}
\end{proof}

\section{Derivation of compact ADI scheme}

In this section, we develop a compact ADI finite difference scheme for the problem (\ref{math201801-01})--(\ref{math201801-03}). Let $\mathit{h}_{x} = \frac{R_{1}-L_{1}}{M_{1}}$, $\mathit{h}_{y} = \frac{R_{2}-L_{2}}{M_{2}}$, and $\mathit{k}_{t} = \frac{T}{N}$ be the spatial and temporal step sizes respectively, where $M_{1}$, $M_{2}$ and $N$ are some given positive integers. Denote $x_{i}=L_{1}+i\mathit{h}_{x}$, $y_{j}=L_{2}+j\mathit{h}_{y}$, $t_{n}=n\mathit{k}_{t}$ for $i = 0, 1, . . . ,M_{1}$, $j = 0, 1, . . . ,M_{2}$ and $n = 0, 1, . . . , N$. We let $u(x_{i},y_{j},t_{n})$ be the exact solution of (\ref{math201801-01})--(\ref{math201801-03}) at the mesh point $(x_{i},y_{j},t_{n})$ and $u_{i,j}^{n}$ represents the solution of an approximating difference scheme at the same mesh point.

Based on Lemma \ref{math201801-05} the Riesz fractional derivative $\frac{\partial^{\gamma}}{\partial|x|^{\gamma}}h(x)$ and the fractional Laplacian operator $-(-\Delta)^{\frac{\gamma}{2}}h(x)$ are equivalent. Thus the two-dimensional Riesz space fractional diffusion equation (\ref{math201801-01}) is in the following form
\begin{equation}\label{math201801-18}
\frac{\partial u(x,y,t)}{\partial
t}=-[\mathcal{C}_{x}(-\Delta_{x})^{\frac{\alpha}{2}}+\mathcal{C}_{y}(-\Delta_{y})^{\frac{\beta}{2}}]u(x,y,t)+s(x,y,t).
\end{equation}
The next stage is to translate each of fractional Riesz derivatives into their corresponding fractional operators at the point $(x_{i},y_{j},t)$. From (\ref{math201801-16}) and (\ref{math201801-17}) in \emph{Appendix A} we have
\begin{equation}\label{math201801-19}
((-\Delta_{x})^{\frac{\alpha}{2}}u)_{i,j}\simeq(-\frac{1}{\mathit{h}_{x}^2}\frac{\delta_{x}^2}{1+\frac{\delta_{x}^2}{12}})^{\frac{\alpha}{2}}u_{i,j},
\end{equation}
and
\begin{equation}\label{math201801-20}
((-\Delta_{y})^{\frac{\beta}{2}}u)_{i,j}\simeq(-\frac{1}{\mathit{h}_{y}^2}\frac{\delta_{y}^2}{1+\frac{\delta_{y}^2}{12}})^{\frac{\beta}{2}}u_{i,j}.
\end{equation}
Substituting (\ref{math201801-19})-(\ref{math201801-20}) into (\ref{math201801-18}) yields
\begin{equation}\label{math201801-21}
\frac{\partial u_{i,j}^{n}}{\partial
t}=-[\mathcal{C}_{x}\mathcal{D}_{\alpha,x}+\mathcal{C}_{y}\mathcal{D}_{\beta,y}]u_{i,j}^{n}+s_{i,j}^{n},
\end{equation}
in which
\begin{equation*}
(-\frac{1}{\mathit{h}_{x}^2}\frac{\delta_{x}^2}{1+\frac{\delta_{x}^2}{12}})^{\frac{\alpha}{2}}=\mathcal{D}_{\alpha,x}, \quad (-\frac{1}{\mathit{h}_{y}^2}\frac{\delta_{y}^2}{1+\frac{\delta_{y}^2}{12}})^{\frac{\beta}{2}}=\mathcal{D}_{\beta,y}.
\end{equation*}
Finally, temporal discretization by Crank-Nicolson method for (\ref{math201801-21}) results in
\begin{equation}\label{math201801-22}
\frac{u_{i,j}^{n+1}-u_{i,j}^{n}}{\mathit{k}_{t}}=-[\mathcal{C}_{x}\mathcal{D}_{\alpha,x}+\mathcal{C}_{y}\mathcal{D}_{\beta,y}]
\frac{u_{i,j}^{n}+u_{i,j}^{n+1}}{2}+\frac{s_{i,j}^{n}+s_{i,j}^{n+1}}{2}.
\end{equation}
After rearrangement and multiplying (\ref{math201801-22}) by $\mathit{k}_{t}$, we have
\begin{equation*}\label{math201801-222}
[1+\frac{\mathit{k}_{t}}{2}(\mathcal{C}_{x}\mathcal{D}_{\alpha,x}+\mathcal{C}_{y}\mathcal{D}_{\beta,y})]u_{i,j}^{n+1}=
[1-\frac{\mathit{k}_{t}}{2}(\mathcal{C}_{x}\mathcal{D}_{\alpha,x}+\mathcal{C}_{y}\mathcal{D}_{\beta,y})]u_{i,j}^{n}
\end{equation*}
\begin{equation}\label{math201801-222}
+\frac{\mathit{k}_{t}}{2}(s_{i,j}^{n}+s_{i,j}^{n+1}).
\end{equation}
We note that the compact finite difference method (\ref{math201801-222}) can be rewritten as the following directional splitting factorization form \cite{Deng2014371}
\begin{equation*}\label{math201801-23}
[1+\frac{\mathit{k}_{t}}{2}\mathcal{C}_{x}\mathcal{D}_{\alpha,x}][1+\frac{\mathit{k}_{t}}{2}
\mathcal{C}_{y}\mathcal{D}_{\beta,y}]u_{i,j}^{n+1}=[1-\frac{\mathit{k}_{t}}{2}\mathcal{C}_{x}\mathcal{D}_{\alpha,x}][1-\frac{\mathit{k}_{t}}{2}
\mathcal{C}_{y}\mathcal{D}_{\beta,y}]u_{i,j}^{n}
\end{equation*}
\begin{equation}\label{math201801-23}
+\frac{\mathit{k}_{t}}{2}(s_{i,j}^{n}+s_{i,j}^{n+1}),
\end{equation}
which introduces an additional perturbation error equal to $\frac{\mathit{k}_{t}}{4}\mathcal{D}_{\alpha,x}\mathcal{D}_{\beta,y}(u_{i,j}^{n+1}-u_{i,j}^{n})$.\\
The additional term is of higher order and do not affect the accuracy of the scheme. In order to simplify the computation, we may re-write the scheme (\ref{math201801-23}) in the Peaceman-Rachford ADI form \cite{Peaceman195541} as
\begin{equation}\label{math201801-24}
[1+\frac{\mathit{k}_{t}}{2}\mathcal{C}_{x}\mathcal{D}_{\alpha,x}]u_{i,j}^{*}=[1-\frac{\mathit{k}_{t}}{2}
\mathcal{C}_{y}\mathcal{D}_{\beta,y}]u_{i,j}^{n}+\frac{\mathit{k}_{t}}{2}(s_{i,j}^{n}+s_{i,j}^{n+1}),
\end{equation}
\begin{equation}\label{math201801-25}
[1+\frac{\mathit{k}_{t}}{2}\mathcal{C}_{y}\mathcal{D}_{\beta,y}]u_{i,j}^{n+1}=[1-\frac{\mathit{k}_{t}}{2}
\mathcal{C}_{x}\mathcal{D}_{\alpha,x}]u_{i,j}^{*}+\frac{\mathit{k}_{t}}{2}(s_{i,j}^{n}+s_{i,j}^{n+1}),
\end{equation}
where $u_{i,j}^{*}$ is an intermediate value.\\
The corresponding algorithm is employed as follows:

(1) First solve on each fixed horizontal slice $y=y_{k}$ $(k=1,2,...,M_{2}-1)$, a set of $M_{1}-1$ equations at the
points $x_{i}$, $i=1,2,...,M_{1}-1$ defined by (\ref{math201801-24}) to obtain the middle solution slice $u_{i,k}^{*}$.

(2) Next alternating the spatial direction, and for each $x = x_{k}$ $(k=1,2,...,M_{1}-1)$
solving a set of $M_{2}-1$ equations defined by (\ref{math201801-25}) at the points $y_{j}$, $j = 1,2,...,M_{2}-1$,
to get $u_{i,j}^{n+1}$.

\section{Stability and convergence analysis}

In this section, we prove consistency and stability for the compact difference scheme (\ref{math201801-23}).
\begin{thm}\label{math201801-26}
The compact difference scheme (\ref{math201801-23}) is unconditionally stable.
\end{thm}
\begin{proof}
To prove the stability of the difference scheme (\ref{math201801-23}), we examine the matrix $(I+S_{x})^{-1}(I-S_{x})\otimes(I+T_{y})^{-1}(I-T_{y})$ that stands as the tensor operator in formula (\ref{math201801-23}).\\
\emph{Appendices A and B}, show that the eigenvalues of matrices $S_{x}$ and $T_{y}$ are positive.
Therefore all the eigenvalues of the matrices  $(I+S_{x})$ and $(I+T_{y})$ are greater than one, and thus this matrices are invertible.
Positivity of eigenvalues of the matrix $S_{x}$ and $T_{y}$ result that every eigenvalue of the matrices of $(I+S_{x})^{-1}(I-S_{x})$ and $(I+T_{y})^{-1}(I-T_{y})$ have the modulates less than one. Therefore, the spectral radius of the matrices $(I+S_{x})^{-1}(I-S_{x})$ and $(I+T_{y})^{-1}(I-T_{y})$ are less than one. $(I+S_{x})^{-1}(I-S_{x})$ and $(I+T_{y})^{-1}(I-T_{y})$ are real and symmetric due to symmetricity of $S_{x}$ and $T_{y}$ (see \emph{Appendix B}). So the norm of the matrices of $(I+S_{x})^{-1}(I-S_{x})$ and $(I+T_{y})^{-1}(I-T_{y})$ are less than one.\\
Hence, the difference scheme (\ref{math201801-23}) is unconditionally stable.
\end{proof}

\begin{thm}\label{math201801-29}
The truncation error of the difference scheme (\ref{math201801-23}) is  $\mathcal{O}(\mathit{h}_{x}^{2\alpha})+\mathcal{O}(h_{y}^{2\beta})+\mathcal{O}(\mathit{k}_{t}^2)$.
\end{thm}

\begin{proof}
Let $u(x_{i},y_{j},t_{n})$ be the exact solution of (\ref{math201801-01})--(\ref{math201801-03}) and $u_{i,j}^{n}$ be the solution of the numerically recurrence scheme (\ref{math201801-23}).
First, we derive the principal error term associated with discretization of the Riesz fractional derivative operators. we note that by considering the arbitrary order $\gamma$ and variable $z$, based on the multiplication property of the order "$\mathcal{O}$", (see Definition \ref{math201801-10}) we have the following relation
\begin{equation*}
(\mathcal{O}(\mathit{h}_{z}^{4}))^{\frac{\gamma}{2}}=\mathcal{O}((\mathit{h}_{z}^{4})^{\frac{\gamma}{2}})=\mathcal{O}(\mathit{h}_{z}^{2\gamma}).
\end{equation*}
By applying the Lemma \ref{math201801-05} on the Eqs. (\ref{math201801-16}), (\ref{math201801-17}), (see \emph{Appendix A}) and smoothness of the exact solution $u$, we have
\begin{equation*}
(\frac{\partial^{\alpha}u}{\partial
|x|^{\alpha}})_{i,j}=-\mathcal{D}_{\alpha,x}u_{i,j}+\mathcal{O}(\mathit{h}_{x}^{2\alpha}),
\end{equation*}
and
\begin{equation*}
(\frac{\partial^{\beta}u}{\partial
|y|^{\beta}})_{i,j}=-\mathcal{D}_{\beta,y}u_{i,j}+\mathcal{O}(\mathit{h}_{y}^{2\beta}).
\end{equation*}
Second, we discuss the local truncation error for scheme (\ref{math201801-23}). We use the two-dimensional case of (\ref{math201801-19})-(\ref{math201801-20}) and Crank-Nicolson scheme to do the discretization in space and time directions, respectively. Substitution in to the expression for (\ref{math201801-01}) yields
\begin{equation*}\label{math201801-30}
[1+\frac{\mathit{k}_{t}}{2}\mathcal{C}_{x}\mathcal{D}_{\alpha,x}+\frac{\mathit{k}_{t}}{2}\mathcal{C}_{y}\mathcal{D}_{\beta,y}]u(x_{i},y_{j},t_{n+1})
\end{equation*}
\begin{equation}\label{math201801-30}
=[1-\frac{\mathit{k}_{t}}{2}\mathcal{C}_{x}\mathcal{D}_{\alpha,x}-\frac{\mathit{k}_{t}}{2}\mathcal{C}_{y}\mathcal{D}_{\beta,y}]u(x_{i},y_{j},t_{n})+R_{i,j}^{n+1}
\end{equation}
where
\begin{equation}\label{math201801-311}
\mid R_{i,j}^{n+1} \mid \leq c \mathit{k}_{t}(\mathcal{O}(\mathit{h}_{x}^{2\alpha})+\mathcal{O}(\mathit{h}_{y}^{2\beta})+\mathcal{O}(\mathit{k}_{t}^2)).
\end{equation}
Finally, we give the global discretization error for numerically approximated scheme (\ref{math201801-23}).
Taking $e_{i,j}^{n}=u(x_{i},y_{j},t_{n})-u_{i,j}^{n}$, and subtracting (\ref{math201801-23}) from (\ref{math201801-30}), yields
\begin{equation}\label{math201801-312}
(I+S_{x})(I+T_{y})e^{n+1}=(I-S_{x})(I-T_{y})e^{n}+R^{n+1}
\end{equation}
where $S_{x}$ and $T_{y}$ are defined in (\ref{math201801-27}) and (\ref{math201801-28}) of \emph{Appendix B}, respectively, and
\begin{equation*}
e^{n}=[e_{1,1}^{n},e_{2,1}^{n},...,e_{M_{1}-1,1}^{n},e_{1,2}^{n},e_{2,2}^{n},...,e_{M_{1}-1,2}^{n},...,
\end{equation*}
\begin{equation*}
e_{1,M_{2}-1}^{n},e_{2,M_{2}-1}^{n},...,e_{M_{1}-1,M_{2}-1}^{n}]^{T},
\end{equation*}
\begin{equation*}
R^{n}=[R_{1,1}^{n},R_{2,1}^{n},...,R_{M_{1}-1,1}^{n},R_{1,2}^{n},R_{2,2}^{n},...,R_{M_{1}-1,2}^{n},...,
\end{equation*}
\begin{equation*}
R_{1,M_{2}-1}^{n},R_{2,M_{2}-1}^{n},...,R_{M_{1}-1,M_{2}-1}^{n}]^{T}.
\end{equation*}
Now from (\ref{math201801-311}) one can write
\begin{equation}\label{math201801-313}
\mid R_{i,j}^{n+1} \mid \leq c \mathit{k}_{t}(\mathcal{O}(\mathit{h}_{x}^{2\alpha})+\mathcal{O}(\mathit{h}_{y}^{2\beta})+\mathcal{O}(\mathit{k}_{t}^2)).
\end{equation}
Since $S_{x}$ and $T_{y}$ commute, then from (\ref{math201801-312})
%Since $S_{x}$ and $T_{y}$ commute, then
\begin{equation*}
e^{n+1}=(I+S_{x})^{-1}(I-S_{x})(I+T_{y})^{-1}(I-T_{y})e^{n}+(I+S_{x})^{-1}(I+T_{y})^{-1}R^{n+1}.
\end{equation*}
With taking the 2-norm on both sides of the above relation, we have
\begin{equation*}
\parallel e^{n+1}\parallel \leq \parallel e^{n}\parallel + \parallel R^{n+1}\parallel.
\end{equation*}
Since from Theorem \ref{math201801-26} one can write
\begin{equation*}
\parallel (I+S_{x})^{-1}(I-S_{x})(I+T_{y})^{-1}(I-T_{y})\parallel
\end{equation*}
\begin{equation*}
\leq \parallel (I+S_{x})^{-1}(I-S_{x})\parallel . \parallel(I+T_{y})^{-1}(I-T_{y})\parallel \leq 1
\end{equation*}
and
\begin{equation*}
\parallel (I+S_{x})^{-1}(I+T_{y})^{-1}\parallel \leq \parallel (I+S_{x})^{-1}\parallel . \parallel(I+T_{y})^{-1}\parallel \leq 1.
\end{equation*}
We use mathematical induction to create the relation between error in final step and errors created in earlier steps, i.e.,
\begin{equation*}
\parallel e^{n+1}\parallel \leq \parallel e^{n}\parallel + \parallel R^{n+1}\parallel \leq \parallel e^{n-1}\parallel+ \parallel R^{n}\parallel + \parallel R^{n+1}\parallel.
\end{equation*}
Since $\parallel e^{0}\parallel=\parallel u(x_{i},y_{j},t_{0})-u_{i,j}^{0}\parallel$ from (\ref{math201801-313}) we conclude that
\begin{equation*}
\parallel e^{n}\parallel \leq \sum_{k=1}^{n}\parallel R^{k}\parallel \leq C(\mathcal{O}(\mathit{h}_{x}^{2\alpha})+
\mathcal{O}(\mathit{h}_{y}^{2\beta})+\mathcal{O}(\mathit{k}_{t}^2)),
\end{equation*}
where $C=n c k_{t}$.\\
It is shown that the solution to (\ref{math201801-01})--(\ref{math201801-03}) can be approximated by numerical scheme (\ref{math201801-23}) with the discretization error $\mathcal{O}(\mathit{h}_{x}^{2\alpha})+
\mathcal{O}(\mathit{h}_{y}^{2\beta})+\mathcal{O}(\mathit{k}_{t}^2)$.
\end{proof}
By Theorems \ref{math201801-26} and \ref{math201801-29} and Lax's equivalence theorem \cite{Smith1985}, the scheme (\ref{math201801-23}) is convergent.

\section{Numerical experiments}

In this section, we will present an example of two dimensional Riesz space fractional diffusion equations. We shall compare the numerical solutions with the exact solutions. To demonstrate the accuracy of preferred method, we have computed not only maximum errors, but also estimated convergence rates separately in spatial and temporal directions.\\
The maximum absolute errors between the exact and the numerical solutions
\begin{equation*}
E_{\infty}(\mathit{h},\mathit{k}_{t})=\max_{i,j} \mid u(x_{i},y_{j},t_{N})-u_{i,j}^{N}\mid,
\end{equation*}
are measured in our examples. Furthermore, the spatial convergence order, denoted by
\begin{equation*}
\mbox{Convergence Rate1}=\log_{2}(E_{\infty}(2\mathit{h},\mathit{k}_{t})/E_{\infty}(\mathit{h},\mathit{k}_{t})),
\end{equation*}
for sufficiently small $\mathit{k}_{t}$, and the temporal convergence order, denoted by
\begin{equation*}
\mbox{Convergence Rate2}=\log_{2}(E_{\infty}(\mathit{h},2\mathit{k}_{t})/E_{\infty}(\mathit{h},\mathit{k}_{t})),
\end{equation*}
when $\mathit{h}$ is sufficiently small, are reporting. The numerical results given by these examples justify our theoretical results.

\begin{ex}
\label{math201801-34}
We consider the following two dimensional Riesz space fractional diffusion equation with the initial and homogeneous Dirichlet boundary conditions:
\begin{eqnarray*}
\frac{\partial u(x,y,t)}{\partial
t}=\mathcal{C}_{x}\frac{\partial^{\alpha}u(x,y,t)}{\partial
|x|^{\alpha}}+\mathcal{C}_{y}\frac{\partial^{\beta}u(x,y,t)}{\partial
|y|^{\beta}}+s(x,y,t),
\end{eqnarray*}
\begin{eqnarray*}
0<t<2 ,\quad 0<x,y<\pi,
\end{eqnarray*}
\begin{eqnarray*}
u(x,y,0)=x^2y^2(\pi-x)(\pi-y), \quad 0\leq x,y \leq \pi,
\end{eqnarray*}
\begin{eqnarray*}
u(0,y,t)=u(\pi,y,t)=u(x,0,t)=u(x,\pi,t)=0, \quad 0\leq t \leq 2,\quad 0\leq x,y \leq \pi,
\end{eqnarray*}
with source function
\begin{equation*}
s(x,y,t)=\frac{\mathcal{C}_{x}y^2(\pi-y)e^{-t}}{2cos(\frac{\pi\alpha}{2})}\Theta(x,\alpha)+
\frac{\mathcal{C}_{y}x^2(\pi-x)e^{-t}}{2cos(\frac{\pi\beta}{2})}\Theta(y,\beta)
%\{\frac{2\pi x^{2-\alpha}}{\Gamma(3-\alpha)}-\frac{6 x^{3-\alpha}}{\Gamma(4-\alpha)}+\frac{\pi^{2}(\pi-x)^{-\alpha} }{\Gamma(1-\alpha)}-\frac{2\pi (\pi-x)^{1-\alpha}}{\Gamma(2-\alpha)}+\frac{2 (\pi-x)^{2-\alpha}}{\Gamma(3-\alpha)}\}
\end{equation*}
\begin{equation*}
-x^2y^2(\pi-x)(\pi-y)e^{-t}
\end{equation*}
where $\Theta(z,\gamma)=\frac{2\pi z^{2-\gamma}}{\Gamma(3-\gamma)}-\frac{6 z^{3-\gamma}}{\Gamma(4-\gamma)}+\frac{\pi^{2}(\pi-z)^{-\gamma} }{\Gamma(1-\gamma)}-\frac{2\pi (\pi-z)^{1-\gamma}}{\Gamma(2-\gamma)}+\frac{2 (\pi-z)^{2-\gamma}}{\Gamma(3-\gamma)}$ and $\mathcal{C}_{x}=\mathcal{C}_{y}=0.25$. The corresponding exact solution is $u(x,y,t)=x^2y^2(\pi-x)(\pi-y)e^{-t}$.

The table \ref{math201801-35} shows maximum absolute errors and related estimated convergence rates with different values for  $\mathit{h}_{x}=\mathit{h}_{y}$ as $0.1\pi$, $0.05\pi$, $0.025\pi$, $0.0125\pi$ and $0.00625\pi$, fixed value $\mathit{k}_{t}=0.001$ whereas Table \ref{math201801-36} presents them with different values for $\mathit{k}_{t}$ as $0.1$, $0.05$, $0.025$, $0.0125$ and $0.00625$ and fixed value $\mathit{h}_{x}=\mathit{h}_{y}=0.001\pi$. Whose fractional derivative orders $\alpha=1.8$ \& $\beta=1.6$ and $\alpha=1.8$ \& $\beta=1.8$ are considered separately in two tables. From Tables \ref{math201801-35} and \ref{math201801-36}, we find the experimental convergence orders are approximately twice the smallest fractional derivative and two in spatial and temporal directions, respectively. The numerical Example results are provided to show that the proposed approximation method is computationally efficient.
%\ref{math201801-34}
\begin{table}
\begin{center}
% table caption is above the table
\caption{The maximum errors and convergence rates for the compact ADI method for solving 2D Riesz space FDE with halved spatial step sizes and $\mathit{k}_{t}= 0.001$}
\label{math201801-35}       % Give a unique label
% For LaTeX tables use
\begin{tabular}{lllll}
\hline\noalign{\smallskip}
             & Max Error       & Convergence   &     Max Error       & Convergence  \\
$\mathit{h}_{x}=\mathit{h}_{y}$&  $\alpha=1.8$, $\beta=1.6$ & Rate & $\alpha=\beta=1.8$ & Rate  \\
\noalign{\smallskip}\hline\noalign{\smallskip}
$0.10000\pi$&      $7.36901e-003$&             &     $5.17481e-003$&             \\
$0.05000\pi$&      $9.30631e-004$&    $2.98519$&     $4.92617e-004$&    $3.39297$\\
$0.02500\pi$&      $1.10146e-004$&    $3.07879$&     $4.45651e-005$&    $3.46648$\\
$0.01250\pi$&      $1.24084e-005$&    $3.15003$&     $3.85100e-006$&    $3.53261$\\
$0.00625\pi$&      $1.35755e-006$&    $3.19224$&     $3.21297e-007$&    $3.58325$\\
\noalign{\smallskip}\hline
\end{tabular}
\end{center}
\end{table}
%%###############################################################################$\alpha=1.8$, $\beta=0.7$, $\mu=1.6$ and $\nu=0.5$
\begin{table}
\begin{center}
% table caption is above the table
\caption{The maximum errors and convergence rates for the compact ADI method for solving 2D Riesz space FDE with halved temporal step sizes and $\mathit{h}_{x}=\mathit{h}_{y}= 0.001\pi$}
\label{math201801-36}       % Give a unique label
% For LaTeX tables use
\begin{tabular}{lllll}
\hline\noalign{\smallskip}
             & Max Error       & Convergence   &     Max Error       & Convergence  \\
$\mathit{k}_{t}$ & $\alpha=1.8$, $\beta=1.6$ & Rate & $\alpha=\beta=1.8$ & Rate  \\
\noalign{\smallskip}\hline\noalign{\smallskip}
$0.10000$&      $8.53972e-003$&             &     $6.78334e-003$&             \\
$0.05000$&      $2.50889e-003$&    $1.76714$&     $2.04295e-003$&    $1.73134$\\
$0.02500$&      $6.85153e-004$&    $1.87255$&     $5.74989e-004$&    $1.82905$\\
$0.01250$&      $1.77936e-004$&    $1.94507$&     $1.53015e-004$&    $1.90986$\\
$0.00625$&      $4.34099e-005$&    $2.03526$&     $3.89268e-005$&    $1.97484$\\
\noalign{\smallskip}\hline
\end{tabular}
\end{center}
\end{table}

\end{ex}

\section{Conclusions}
In the present work, a high order compact ADI method for solving the two dimensional Riesz space fractional diffusion equation has been established. The method is spatially twice the smallest fractional derivative- and temporally second-order accuracy. It is shown through a matrix analysis that it is unconditionally stable. Numerical results are provided to verify the accuracy and efficiency of the preferred method.
%===============================================================================================================
\section*{Acknowledgement}
The authors express their deep gratitude to the Research Council of University of Mohaghegh Ardabili for funding this research work.
%===============================================================================================================

\section*{Appendix A}
We consider the fourth-order compact approximations for the second-order derivative operators based on Lemma \ref{math201801-11} (also see \cite{Ding2009600})
\begin{equation}\label{math201801-16}
(\frac{\partial^{2}u}{\partial
x^{2}})_{i,j}= \frac{1}{\mathit{h}_{x}^2}\frac{\delta_{x}^2}{1+\frac{\delta_{x}^2}{12}}u_{i,j}+\mathcal{O}(\mathit{h}_{x}^4),\quad \mbox{for}\quad i=0,1,...,M_{1}\quad \mbox{and fix}\quad j
\end{equation}
\begin{equation}\label{math201801-17}
(\frac{\partial^{2}u}{\partial
y^{2}})_{i,j}= \frac{1}{\mathit{h}_{y}^2}\frac{\delta_{y}^2}{1+\frac{\delta_{y}^2}{12}}u_{i,j}+\mathcal{O}(\mathit{h}_{y}^4),\quad \mbox{for}\quad j=0,1,...,M_{2}\quad \mbox{and fix}\quad i
\end{equation}
where $\delta_{x}^2$ and $\delta_{y}^2$ are the standard second-order central difference operators in x- and y- directory respectively.
If the boundary values at $i=0$ and $i=M_{1}$, $j>0$, are known, these $(M_{1}-1)$ equation for $i=1,2,...,M_{1}-1$ can be written in matrix form
\begin{equation*}
((\frac{1}{\mathit{h}_{x}^2}\frac{\delta_{x}^2}{1+\frac{\delta_{x}^2}{12}})_{i,j})=A_{x}^{-1}B_{x},\quad i,j=1,2,...,M_{1}-1,
\end{equation*}
where $A_{x}=\frac{\mathit{h}_{x}^2}{12}diag(1,10,1)$ and $B_{x}=diag(1,-2,1)$ are tridiagonal matrices of $M_{1}-1$ order.
And if the boundary values at $j=0$ and $j=M_{2}$, $i>0$, are known, these $(M_{2}-1)$ equation for $j=1,2,...,M_{2}-1$ can be written in matrix form
\begin{equation*}
((\frac{1}{\mathit{h}_{y}^2}\frac{\delta_{y}^2}{1+\frac{\delta_{y}^2}{12}})_{i,j})=A_{y}^{-1}B_{y},\quad i,j=1,2,...,M_{2}-1,
\end{equation*}
where $A_{y}=\frac{\mathit{h}_{y}^2}{12}diag(1,10,1)$ and $B_{y}=diag(1,-2,1)$ are tridiagonal matrices of $M_{2}-1$ order.

Referring to the Lemma \ref{math201801-08} our achievement on that the eigenvalues of the matrix of the $\frac{-1}{\mathit{h}_{x}^2}\frac{\delta_{x}^2}{1+\frac{\delta_{x}^2}{12}}$ operator is as follows:
\begin{align*}
\lambda_{j}&=[\frac{\mathit{h}_{x}^2}{12}(10+2\cos(j\pi/M_{1}))]^{-1}\times (-1)\times[-2+2\cos(j\pi/M_{1})]\\
&=\frac{12}{\mathit{h}_{x}^2}\frac{2-2\cos (j\pi/M_{1})}{12-2+2\cos(j\pi/M_{1})}\\
&=\frac{12}{\mathit{h}_{x}^2}\frac{4\sin^{2}(j\pi/2M_{1})}{12-4\sin^{2}(j\pi/2M_{1})}\\
&=\frac{12}{\mathit{h}_{x}^2}\sin^{2}(j\pi/2M_{1})(3-\sin^{2} (j\pi/2M_{1}))^{-1}
\end{align*}
Since the eigenvalues of matrix $(\frac{-1}{\mathit{h}_{x}^2}\frac{\delta_{x}^2}{1+\frac{\delta_{x}^2}{12}})$ are distinct positive. So there is a pair of matrices $\mathbf{D}_{x}$ and $P$ that $\mathbf{D}_{x}$ is a diagonal matrix which members are eigenvalues of matrix $-A_{x}^{-1}B_{x}$ and the columns of the matrix $P$ are eigenvectors corresponding to the these eigenvalues and we have
\begin{equation*}
-A_{x}^{-1}B_{x}=P\mathbf{D}_{x}P^{-1}
\end{equation*}
And similarly, in the direction of the second axis, there are pair matrices $\mathbf{D}_{y}$ and $Q$ which have the following relation
\begin{equation*}
-A_{y}^{-1}B_{y}=Q\mathbf{D}_{y}Q^{-1}.
\end{equation*}
As respects the eigenvalues of the matrices $-A_{x}^{-1}B_{x}$ and $-A_{y}^{-1}B_{y}$ are positive and distinct, and the matrices $A_{x}$, $B_{x}$, $A_{y}$ and $B_{y}$ are all symmetric, so the matrices $-A_{x}^{-1}B_{x}$ and $-A_{y}^{-1}B_{y}$ are symmetric positive definite.

\section*{Appendix B}
In this section, the matrix form of the operators $\frac{\mathit{k}_{t}}{2}\mathcal{C}_{x}\mathcal{D}_{\alpha,x}$ and $\frac{\mathit{k}_{t}}{2}\mathcal{C}_{y}\mathcal{D}_{\beta,y}$, which is displayed by $S_{x}$ and $T_{y}$ respectively, is represented.
\begin{equation}\label{math201801-27}
S_{x}=m\{\frac{\mathit{k}_{t}}{2}\mathcal{C}_{x}\mathcal{D}_{\alpha,x}\}=
\frac{\mathit{k}_{t}}{2}\mathcal{C}_{x}P\mathbf{D}^{\frac{\alpha}{2}}_{x}P^{-1}
\end{equation}
where $\mathbf{D}_{x}=diag(\lambda_{1}\quad \lambda_{2}\quad ...\quad \lambda_{M_{1}-1})$ and $P=(x_{1}\quad x_{2}\quad ...\quad x_{M_{1}-1})$ in which
\begin{equation*}
x_{i}=\left(
\begin{array}{c}
  \sin(1i\pi/M_{1}) \\
  \sin(2i\pi/M_{1}) \\
  \sin(3i\pi/M_{1}) \\
  \vdots \\
  \sin((M_{1}-1)i\pi/M_{1}) \\
\end{array}
\right),
\end{equation*}
\begin{equation*}
\lambda_{i}=\frac{12\sin^{2}(i\pi/2M_{1})}{\mathit{h}_{x}^2(3-\sin^{2} (i\pi/2M_{1}))}, \quad i=1,2,...,M_{1}-1,
\end{equation*}
and similarly
\begin{equation}\label{math201801-28}
T_{y}=m\{\frac{\mathit{k}_{t}}{2}\mathcal{C}_{y}\mathcal{D}_{\beta,y}\}=
\frac{\mathit{k}_{t}}{2}\mathcal{C}_{y}Q\mathbf{D}^{\frac{\beta}{2}}_{y}Q^{-1}
\end{equation}
where $\mathbf{D}_{y}=diag(\gamma_{1}\quad \gamma_{2}\quad ...\quad \gamma_{M_{2}-1})$ and $Q=(y_{1}\quad y_{2}\quad ...\quad y_{M_{2}-1})$ in which
\begin{equation*}
y_{j}=\left(
\begin{array}{c}
  \sin(1j\pi/M_{2}) \\
  \sin(2j\pi/M_{2}) \\
  \sin(3j\pi/M_{2}) \\
  \vdots \\
  \sin((M_{2}-1)j\pi/M_{2}) \\
\end{array}
\right),
\end{equation*}
\begin{equation*}
\gamma_{j}=\frac{12\sin^{2}(j\pi/2M_{2})}{\mathit{h}_{y}^2(3-\sin^{2} (j\pi/2M_{2}))}, \quad j=1,2,...,M_{2}-1,
\end{equation*}
By attention to the positivity of the eigenvalues of the matrices $-A_{x}^{-1}B_{x}$ and $-A_{y}^{-1}B_{y}$ for every orders $M_{1}-1$ and $M_{2}-1$ respectively, eigenvalues of matrices of $S_{x}$ and $T_{y}$ are positive and we have from Appendix $A$ that the two matrices $-A_{x}^{-1}B_{x}$ and $-A_{y}^{-1}B_{y}$ are real and symmetric therefore the two matrices $S_{x}$ and $T_{y}$ are real and symmetric. Moreover, note that the two matrices $S_{x}$ and $T_{y}$ commute, i.e.
\begin{equation*}
S_{x}\otimes T_{y}=T_{y}\otimes S_{x}.
\end{equation*}

% ------------------------------------------------------------------------

\begin{thebibliography}{1}
%1
\bibitem{Alikhanov2015424} A. A. Alikhanov, \textit{A new difference scheme for the time fractional diffusion equation.} J. Comput. Phys. \textbf{280} (2015), 424--438.
%2
\bibitem{Baeumer20011543} B. Baeumer, D. A. Benson, M. M. Meerschaert and S. W. Wheatcraft, \textit{Subordinated advection-dispersion equation for contaminant transport.} Water Resour. Res. \textbf{37(6)} (2001), 1543--1550.
%3
\bibitem{Benson20001403} D. A. Benson, S. W. Wheatcraft and M. M. Meerschaert, \textit{Application of a fractional advection-dispersion equation.} Water Resour. Res. \textbf{36(6)} (2000), 1403--1412.
%4
\bibitem{Borhanifar2015466} A. Borhanifar and S. Valizadeh, \textit{Mittag-Leffler-Pad\'{e} approximations for the numerical solution of space and time fractional diffusion equations.} International Journal of Applied Mathematical Research \textbf{4(4)} (2015), 466--480.
%5
\bibitem{Bu201426} W. Bu, Y. Tang and J. Yang, \textit{Galerkin finite element method for two-dimensional Riesz space fractional diffusion equations.} J. Comput. Phys. \textbf{276} (2014), 26--38.
%6
\bibitem{Cartea2007749} A. Cartea and D. del-Castillo-Negrete, \textit{Fractional diffusion models of option prices in markets with jumps.} Physica A \textbf{374(2)} (2007), 749--763.
%7
\bibitem{Celik20121743} C. Celik and M. Duman, \textit{Crank-Nicolson method for the fractional diffusion equation with the Riesz fractional derivative.} J. Comput. Phys. \textbf{231(4)} (2012), 1743--1750.
%8
\bibitem{Cui20097792} M. Cui, \textit{Compact finite difference method for the fractional diffusion equation.} J. Comput. Phys. \textbf{228(20)} (2009), 7792--7804.
%9
\bibitem{Deng2014371} W. Deng and M. Chen, \textit{Efficient numerical algorithms for three-dimensional fractional partial differential equations.} J. Comput. Math. \textbf{32(4)} (2013), 371--391. doi:10.4208/jcm.1401-m3893
%10
\bibitem{Dentz2004155} M. Dentz, A. Cortis, H. Scher and B. Berkowitz, \textit{Time behavior of solute transport in heterogeneous media: transition from anomalous to normal transport.} Adv. Water Resour. \textbf{27(2)} (2004), 155--173.
%11
\bibitem{Ding2009600} H. Ding and Y. Zhang, \textit{A new difference scheme with high accuracy and absolute stability for solving convection-diffusion equations.} J. Comput. Appl. Math. \textbf{230(2)} (2009), 600--606.
%12
\bibitem{Ding20121135} H. F. Ding and Y. X. Zhang, \textit{New numerical methods for the Riesz space fractional partial differential equations.} Comput. Math. Appl. \textbf{63(7)} (2012), 1135--1146.
%13
\bibitem{Erdelyi1956} A. Erd\'{e}lyi, \textit{Asymptotic Expansions.} 3rd Edition,
Courier Corporation, 1956.
%14
\bibitem{Gal2010} N. Gal and D. Weihs, \textit{Experimental evidence of strong anomalous diffusion in living cells.} Phys. Rev. E \textbf{81(2)} (2010), 020903.
%15
\bibitem{Gao2015487} G. H. Gao and H. W. Sun, \textit{Three-point combined compact alternating direction implicit difference schemes for two-dimensional time-fractional advection-diffusion equations.} Commun. Comput. Phys. \textbf{17(2)} (2015), 487--509.
%16
\bibitem{Gorenflo1998167} R. Gorenflo and F. Mainardi, \textit{Random walk models for space-fractional diffusion processes.} Fract. Calc. Appl. Anal. \textbf{1(2)} (1998), 167--191.
%17
\bibitem{Hilfer2000} R. Hilfer, \textit{Applications of fractional calculus in physics.} 2rd Edition,
world scientific, 2000.
%18
\bibitem{Hunter2004} J. K. Hunter, \textit{Asymptotic analysis and singular perturbation theory.} Department of Mathematics, University of California at Davis, 2004.
%19
\bibitem{Liu2004209} F. Liu, V. Anh and I. Turner, \textit{Numerical solution of the space fractional Fokker-Planck equation.} J. Comput. Appl. Math. \textbf{166(1)} (2004), 209--219.
%20
\bibitem{Magin2006} R. L. Magin, \textit{Fractional calculus in bioengineering.} Begell House Publisher, Connecticut, 2006.
%21
\bibitem{Miller1993} K. S. Miller and B. Ross, \textit{An Introduction to the Fractional Calculus and Fractional Differential Equations.} Wiley, 1993.
%22
\bibitem{Meyer2004} C. D. Meyer, \textit{Matrix Analysis and Applied Linear Algebra.} SIAM, 2004.
%23
\bibitem{Ortigueira20061} M. D. Ortigueira, \textit{Riesz potential operators and inverses via fractional centred derivatives.} International Journal of Mathematics and Mathematical Sciences Hindawi Publishing Corporation (2006), 1--12.
%24
\bibitem{Peaceman195541} D. W. Peaceman and H. H. Rachford, \textit{The numerical solution of parabolic and elliptic differential equations.} J. Soc. Indust. Appl. Math. \textbf{3(1)} (1955), 28--41.
%25
\bibitem{Podlubny1998} I. Podlubny, \textit{Fractional differential equations: an introduction to fractional derivatives, fractional differential equations, to methods of their solution and some of their applications.} Elsevier, 1998.
%26
\bibitem{Popolizio20131975} M. Popolizio, \textit{A matrix approach for partial differential equations with Riesz space fractional derivatives.} Eur. Phys. J. Special Topics \textbf{222(8)} (2013), 1975--1985.
%27
\bibitem{Sadek20121741} G. Sadek, Q. Liu, and M. A. Ragusa, \textit{A new regularity criterion for the nematic liquid crystal flows.} Applicable Analysis \textbf{91(9)} (2012), 1741--1747.
%28
\bibitem{Sadek20161271} G. Sadek and M. A. Ragusa, \textit{Logarithmically improved regularity criterion for the Boussinesq equations in Besov spaces with negative indices.} Applicable Analysis \textbf{95(6)} (2016), 1271--1279.
%29
\bibitem{Sagi2012} Y. Sagi, M. Brook, I. Almog and N. Davidson, \textit{Observation of anomalous diffusion and fractional self-similarity in one dimension.} Phys. Rev. Lett. \textbf{108(9)} (2012), 093002.
%30
\bibitem{Kilbas1993} A. A. Kilbas and O. I. Marivhev and S. G. Samko, \textit{Fractional Integrals and Derivatives: Theory and Applications.} Gordon and Breach, 1993.
%31
\bibitem{Smith1985} G. D. Smith, \textit{Numerical solution of partial differential equations: finite difference methods.} Oxford university press, 1985.
%32
\bibitem{Tadjeran2007813} C. Tadjeran and M. M. Meerschaert, \textit{A second-order accurate numerical method for the two-dimensional fractional diffusion equation.} J. Comput. Phys. \textbf{220(2)} (2007), 813--823.
%33
\bibitem{Wang20117830} H. Wang and K. Wang, \textit{An $O(N log^{2} N)$ alternating-direction finite difference method for two-dimensional fractional diffusion equations.} J. Comput. Phys. \textbf{230(21)} (2011), 7830--7839.
%34
\bibitem{Yang2010200} Q. Yang, F. Liu and I. Turner, \textit{Numerical methods for fractional partial differential equations with Riesz space fractional derivatives.} Appl. Math. Model. \textbf{34(1)} (2010), 200--218.
%35
\bibitem{Zaslavsky2002461} G. M. Zaslavsky, \textit{Chaos, fractional kinetics, and anomalous transport.} Phys. Rep. \textbf{371(6)} (2002), 461--580.
%36
\bibitem{Zeng20142599} F. Zeng, F. Liu, C. Li, K. Burrage, I. Turner and V. Anh, \textit{A Crank-Nicolson ADI spectral method for a two-dimensional Riesz space farctional nonlinear reaction-diffusion equation.} SIAM J. Numer. Anal. \textbf{52(6)} (2014), 2599--2622.
%37
\bibitem{Zhang20118713} Y. N. Zhang and Z. Z. Sun, \textit{Alternating direction implicit schemes for the two-dimensional fractional sub-diffusion equation.} J. Comput. Phys. \textbf{230(24)} (2011), 8713--8728.
%38
\bibitem{Zhang2014266} Y. Zhang and H. Ding, \textit{Improved matrix transform method for the Riesz space fractional reaction dispersion equation.} J. Comput. Appl. Math. \textbf{260} (2014), 266--280.

\end{thebibliography}
\end{document}